\documentclass[reqno,11pt]{amsart}
\usepackage{amssymb}
\usepackage{amsmath}
\usepackage{stmaryrd}
\usepackage{graphicx}
\usepackage{color}
\usepackage{cite}
\usepackage[font=small]{caption}
\usepackage{cleveref}
\usepackage{tikz}
\usepackage{pgfplots}
\pgfplotsset{compat=1.10}
\usepgfplotslibrary{fillbetween}
\usetikzlibrary{patterns}
\newtheorem{theorem}{Theorem}[section]
\newtheorem{corollary}[theorem]{Corollary}
\newtheorem{lemma}[theorem]{Lemma}
\newtheorem{proposition}[theorem]{Proposition}

\newtheorem{remark}[theorem]{Remark}

\newcommand{\R}{\mathbb{R}}

\newcommand{\rd}{\mathrm{d}}

\definecolor{cadmiumgreen}{rgb}{0.0, 0.42, 0.24}

\numberwithin{equation}{section}
\numberwithin{figure}{section}
\begin{document}

\title[Free Boundary Transmission MEMS Model]{Stationary States to a Free Boundary Transmission Problem for an Electrostatically Actuated Plate}

\author{Philippe Lauren\c{c}ot}
\address{Institut de Math\'ematiques de Toulouse, UMR~5219, Universit\'e de Toulouse, CNRS \\ F--31062 Toulouse Cedex 9, France}
\email{laurenco@math.univ-toulouse.fr}
\author{Christoph Walker}
\address{Leibniz Universit\"at Hannover\\ Institut f\" ur Angewandte Mathematik \\ Welfengarten 1 \\ D--30167 Hannover\\ Germany}
\email{walker@ifam.uni-hannover.de}
%
%
\date{\today}
%
\begin{abstract}
A two-dimensional free boundary transmission problem arising in the modeling of an electrostatically actuated plate is considered and a representation formula for the derivative of the associated electrostatic energy with respect to the deflection of the plate is derived. The latter paves the way for the construction of energy minimizers and also provides the Euler-Lagrange equation satisfied by these minimizers. A by-product is the monotonicity of the electrostatic energy with respect to the deflection.
\end{abstract}
%
\maketitle
%
\section{Introduction}\label{IMR}

We consider a model for a microelectromechanical system (MEMS) featuring an elastic, electrostatically actuated plate with positive thickness as introduced in~\cite{LW18}. More precisely, given a finite interval $D:=(-L,L)$ with $L>0$, let the function $u\in C(\bar{D},[-H,\infty))$ with $u(\pm L)=0$ measure the deflection from rest of the lower part of an elastic plate with thickness $d>0$, clamped at its boundaries and suspended above a fixed ground plate, the latter being represented by $D$ and located at $z=-H$ with $H>0$ (see Figure~\ref{Fig1}). The deflected elastic plate is then
\begin{equation*}
	\Omega_2(u):= \left\{ (x,z)\in D\times \mathbb{R}\,:\, u(x)<  z <  u(x)+d\right\}\,, 
\end{equation*}
while the region between the ground plate and the deflected elastic plate is
\begin{equation*}
	\Omega_1(u) := \left\{ (x, z)\in D\times \mathbb{R} \,:\, -H< z< {u}(x)\right\}\,.
\end{equation*}
The two regions are separated by the interface
\begin{equation*}
	\Sigma(u) := \left\{ (x,z)\in D\times \mathbb{R}\,:\, z=  u(x)>-H \right\}\,.
\end{equation*}
The deflection of the plate being triggered by electrostatic actuation, the total energy of the device is
\begin{subequations}\label{E}
\begin{equation}
	E(u):= E_m(u)+E_e(u)
\end{equation}
with  mechanical energy  $E_m(u)$  and   electrostatic energy $E_e(u)$. The former is given by
\begin{equation}\label{E1}
	E_m(u):=\frac{\beta}{2}\|\partial_x^2u\|_{L_2(D)}^2 +\left(\frac{\tau}{2}+\frac{a}{4}\|\partial_x u\|_{L_2(D)}^2\right)\|\partial_x u\|_{L_2(D)}^2
\end{equation}
with $\beta>0$ and $a, \tau\ge 0$, taking into account bending and external stretching effects of the elastic plate. The electrostatic energy
\begin{equation}\label{E2}
E_e(u):=-\frac{1}{2}\int_{\Omega(u)} \sigma \vert\nabla \psi_u\vert^2\,\rd (x,z)
\end{equation}
\end{subequations}
involves the electrostatic potential $\psi_u$ in the subdomain 
\begin{equation*}
 \Omega(u) := \left\{ (x,z)\in D\times \mathbb{R} \,:\, -H<  z <  u(x)+d \right\} = \Omega_1(u)\cup \Omega_2(u)\cup \Sigma(u)
\end{equation*}
of $D\times (-H,\infty)$. The electrostatic potential $\psi_u$ is the solution to the transmission problem
\begin{subequations}\label{psi}
	\begin{align}
		\mathrm{div}(\sigma\nabla\psi_u) &=0 \quad\text{in }\ \Omega(u)\,,\label{a1a}\\
		\llbracket \psi_u \rrbracket =\llbracket \sigma\nabla \psi_u \rrbracket \cdot \mathbf{n}_{ \Sigma(u)} &=0 \quad\text{on }\ \Sigma(u)\,,\label{a1b}\\
		\psi_u&=h_u\quad\text{on }\ \partial\Omega(u)\,,\label{a1c}
	\end{align}
\end{subequations}
where $\llbracket \cdot \rrbracket$ denotes the (possible) jump across the interface $\Sigma(u)$; that is, 
\begin{equation*}
	\llbracket f \rrbracket(x,u(x)) := f|_{\Omega_1(u)}(x,u(x)) - f|_{\Omega_2(u)}(x,u(x))\,, \qquad x\in D\,,
\end{equation*}
whenever meaningful for a function $f:\Omega(u)\to\mathbb{R}$. Moreover,
\begin{equation}
	\sigma := \sigma_1 \mathbf{1}_{\Omega_1(u)} + \sigma_2 \mathbf{1}_{\Omega_2(u)} \label{sigma}
\end{equation}
involves the material dependent constant permittivities $\sigma_2,\sigma_1>0$. The unit normal vector field to $\Sigma(u)$  (pointing into $\Omega_2(u)$) is
\begin{equation*}
\mathbf{n}_{ \Sigma(u)}:=\frac{(-\partial_x u, 1)}{\sqrt{1+(\partial_x u)^2}}\,.
\end{equation*}
\begin{figure}
	\begin{tikzpicture}[scale=0.9]
		\draw[black, line width = 0.8pt, dashed] (-7,0)--(7,0);
		\draw[black, line width = 0.8pt, dashed] (-7,-0.5)--(7,-0.5);
		\draw[black, line width = 2pt] (-7,0)--(-7,-2.5);
		\draw[black, line width = 2pt] (7,-2.5)--(7,0);
		\draw[black, line width = 2pt] (-7,-2.5)--(7,-2.5);
		\draw[blue, line width = 2pt] plot[domain=-7:1] (\x,{-0.75-0.75*cos((pi*(\x-1)/8) r)});
		\draw[blue, line width = 2pt] plot[domain=1:7] (\x,{-0.75-0.75*cos((pi*(\x-1)/6) r)});
		\draw[blue, line width = 2pt] plot[domain=-7:1] (\x,{-1.25-0.75*cos((pi*(\x-1)/8) r)});
		\draw[blue, line width = 2pt] plot[domain=1:7] (\x,{-1.25-0.75*cos((pi*(\x-1)/6) r)});
		\draw[blue, line width = 1pt, arrows=->] (2,-0.5)--(2,-1.9); 
		\node at (2.2,-1) {${\color{blue} u}$};
		\node at (-5,-1.5) {${\color{blue} \Omega_1(u)}$};
		\node at (-2,0.5) {${\color{blue} \Omega_2(u)}$};
		\draw (-2.6,0.4) edge[->,bend right,line width = 1pt] (-3.7,-0.85);
		\node at (0,-3.25) {$D$};
		\node at (5.75,-1.75) {{\color{blue} $\Sigma(u)$}};
		\draw (5.25,-1.75) edge[->,bend left, line width = 1pt] (3.5,-1.475);
		\node at (-7.8,1) {$z$};
		\draw[black, line width = 1pt, arrows = ->] (-7.5,-3)--(-7.5,1);
		\node at (-8,-2.5) {$-H$};
		\draw[black, line width = 1pt] (-7.6,-2.5)--(-7.4,-2.5);
		\node at (-7.8,-0.5) {$0$};
		\draw[black, line width = 1pt] (-7.6,-0.5)--(-7.4,-0.5);
		\node at (-7.8,0) {$d$};
		\draw[black, line width = 1pt] (-7.6,0)--(-7.4,0);
		\node at (-7,-3.25) {$-L$};
		\node at (7,-3.25) {$L$};
		\draw[black, line width = 1pt, arrows = <->] (-7,-2.75)--(7,-2.75);
	\end{tikzpicture}
	\caption{Geometry of $\Omega(u)$ for a state $u\in \mathcal{S}$ with empty coincidence set.}\label{Fig1}
\end{figure}
\noindent As for the boundary values in \eqref{a1c} we assume the particular form 
\begin{subequations}\label{bobbybrown}
\begin{equation}\label{exx1}
h_u(x,z) :=\zeta(z- u(x)+1)\,, \qquad (x,z)\in \bar{D}\times [-H,\infty)\,,
\end{equation}
where $V>0$ and 
\begin{equation}\label{z1}
\zeta\in C^2(\R)\,,\qquad \zeta|_{(-\infty,1]}\equiv 0\,,\qquad \zeta|_{[1+d,\infty)} \equiv V\,.
\end{equation}
\end{subequations}
For instance, $\zeta(r):=V\min\{1,(r-1)^m/d^m\}$ for $r>1$ and $m>2$ and $\zeta\equiv 0$ on $(-\infty,1]$ is a possible choice.
Note that
\begin{equation*}
h_u(x,-H)=0\,, \quad h_u(x,u(x)+d)=V\,,\qquad x\in D\,;
\end{equation*}
that is,  the ground plate and the top of the elastic plate are kept at different constant potentials. Let us emphasize that we explicitly allow that the elastic plate touches upon the ground plate when $u$ reaches the value $-H$ somewhere, a situation corresponding to a nonempty \textit{coincidence set}
\begin{equation}
	\mathcal{C}(u) := \{x\in D\,:\, u(x)=-H\}\,, \label{CS}
\end{equation}
as depicted in Figure~\ref{Fig2}. In this case, the region $\Omega_1(u)$ is not connected and its boundary features cusps, so that its connected components are not Lipschitz domains.

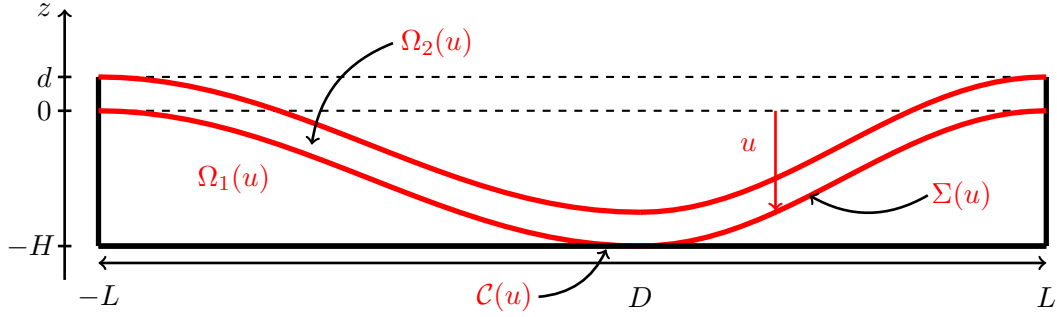
\begin{figure}
		\begin{tikzpicture}[scale=0.9]
		\draw[black, line width = 0.8pt, dashed] (-7,0)--(7,0);
		\draw[black, line width = 0.8pt, dashed] (-7,-0.5)--(7,-0.5);
		\draw[black, line width = 2pt] (-7,0)--(-7,-2.5);
		\draw[black, line width = 2pt] (7,-2.5)--(7,0);
		\draw[black, line width = 2pt] (-7,-2.5)--(7,-2.5);
		\draw[red, line width = 2pt] plot[domain=-7:1] (\x,{-1-cos((pi*(\x-1)/8) r)});
		\draw[red, line width = 2pt] plot[domain=1:7] (\x,{-1-cos((pi*(\x-1)/6) r)});
		\draw[red, line width = 2pt] plot[domain=-7:1] (\x,{-1.5-cos((pi*(\x-1)/8) r)});
		\draw[red, line width = 2pt] plot[domain=1:7] (\x,{-1.5-cos((pi*(\x-1)/6) r)});
		\draw[red, line width = 1pt, arrows=->] (3,-0.5)--(3,-1.975);
		\node at (2.6,-1) {${\color{red} u}$};
		\node at (-5,-1.5) {${\color{red} \Omega_1(u)}$};
		\node at (-2,0.5) {${\color{red} \Omega_2(u)}$};
		\draw (-2.65,0.5) edge[->,bend right,line width = 1pt] (-3.85,-1);
		\node at (1,-3.25) {$D$};
		\node at (5.75,-1.75) {{\color{red} $\Sigma(u)$}};
		\draw (5.25,-1.75) edge[->,bend left, line width = 1pt] (3.5,-1.75);
		\node at (-7.8,1) {$z$};
		\draw[black, line width = 1pt, arrows = ->] (-7.5,-3)--(-7.5,1);
		\node at (-8,-2.5) {$-H$};
		\draw[black, line width = 1pt] (-7.6,-2.5)--(-7.4,-2.5);
		\node at (-7.8,-0.5) {$0$};
		\draw[black, line width = 1pt] (-7.6,-0.5)--(-7.4,-0.5);
		\node at (-7.8,0) {$d$};
		\draw[black, line width = 1pt] (-7.6,0)--(-7.4,0);
		\node at (-7,-3.25) {$-L$};
		\node at (7,-3.25) {$L$};
		\draw[black, line width = 1pt, arrows = <->] (-7,-2.75)--(7,-2.75);
		\node at (-1,-3.25) {${\color{red} \mathcal{C}(u)}$};
		\draw (-0.5,-3.25) edge[->,bend right, line width = 1pt] (0.5,-2.55);
		\draw[black, line width = 2pt] (-7,-2.5)--(7,-2.5);
	\end{tikzpicture}
	\caption{Geometry of $\Omega(u)$ for a state $u\in \bar{\mathcal{S}}$ with non-empty coincidence set.}\label{Fig2}
\end{figure}

In this research we shall be interested in minimizers of the total energy $E$ which correspond to stationary states of the MEMS device. More precisely, we shall show the existence of minimizers and derive the corresponding Euler-Lagrange equation they satisfy, which, due to the nature of the problem, is a variational inequality.  Obviously, the main difficulty in this regard is related to the electrostatic energy $E_e$ and the associated transmission problem~\eqref{psi} for the electrostatic potential. The latter was investigated in~\cite{LW21} for deflections belonging to the set
\begin{equation*}
\bar{\mathcal{S}} := \{u\in H^2(D) \cap H_0^1(D)\,:\, u\ge -H \text{ in } D \;\text{ and }\; \pm\llbracket\sigma\rrbracket  \partial_x u(\pm L) \le 0 \}
\end{equation*}
with $\llbracket\sigma\rrbracket =\sigma_1-\sigma_2 $. More precisely, the following result is shown in \cite{LW21}. 

\begin{theorem}{\bf\cite[Theorem 1.1]{LW21}}\label{Thm1}
Suppose \eqref{bobbybrown}. 
\begin{itemize}
\item[(a)] For each $u\in \bar{\mathcal{S}}$, there is a unique variational solution $\psi_u \in h_{u}+H_{0}^1(\Omega(u))$ to \eqref{psi}.  Moreover, $\psi_{u,1}:= \psi_{u}|_{\Omega_1(u)} \in H^2(\Omega_1(u))$ and $\psi_{u,2} := \psi_{u}|_{\Omega_2(u)} \in H^2(\Omega_2(u))$, and $\psi_{u}$  is a strong solution to the transmission problem~\eqref{psi}.
\item[(b)] Given $\kappa>0$, there is $c(\kappa)>0$ such that, for every $u\in\bar{\mathcal{S}}$ satisfying $\|u\|_{H^2(D)}\le \kappa$, 
\begin{equation*}
\|\psi_u\|_{H^1(\Omega(u))} + \|\psi_{u,1}\|_{H^2(\Omega_1(u))} + \|\psi_{u,2}\|_{H^2(\Omega_2(u))} \le c(\kappa)\,.
\end{equation*}	
\end{itemize}
\end{theorem}

The $H^2$-regularity of the electrostatic potential $\psi_u$ provided by \Cref{Thm1} is then the basis for deriving the existence of minimizers of the total energy $E$. We shall look for minimizers with clamped boundary conditions; that is, minimizers in the closed convex subset
\begin{equation*}
	\bar{\mathcal{S}}_0 := \{u\in H^2(D) \cap H_0^1(D)\,:\, u\ge -H \text{ in } D \;\text{ and }\;  \partial_x u(\pm L)= 0 \}
\end{equation*}
of $H^2(D)$. We denote by $\partial\mathbb{I}_{\bar{\mathcal{S}}_0}$ the subdifferential of the indicator function $\mathbb{I}_{\bar{\mathcal{S}}_0}$.  Our main result then reads: 

\begin{theorem}\label{Thm3}
	Assume $a>0$ or~$\llbracket\sigma\rrbracket<0$, and let \eqref{bobbybrown} be satisfied. Then, the total energy~$E$ has at least one minimizer in $\bar{\mathcal{S}}_0$. Moreover, any minimizer $u_*\in \bar{\mathcal{S}}_0$ of $E$ in $\bar{\mathcal{S}}_0$ with
\begin{equation}\label{B}
E(u_*)=\min_{\bar{\mathcal{S}}_0}E
\end{equation}
is an $H^2$-weak solution to the variational inequality
\begin{equation}
\beta\partial_x^4u_*-(\tau+a\|\partial_x u_*\|_{L_2(D)}^2)\partial_x^2 u_*+\partial\mathbb{I}_{\bar{\mathcal{S}}_0}(u_*) \ni -g(u_*) \;\;\text{ in }\;\; D\,; \label{bennygoodman}
\end{equation}
 that is, 
$$
\int_D \left\{\beta\partial_x^2 u_*\,\partial_x^2 (w-u_*)+\big[\tau+a\|\partial_x u_*\|_{L_2(D)}^2\big]\partial_x u_*\, \partial_x(w-u_*)\right\}\,\rd x\ge -\int_D g(u_*) (w-u_*)\, \rd x  
$$
for all $w\in {\bar{\mathcal{S}}_0}$. The function $g(u)\in L_2(D)$  is for $u\in \bar{\mathcal{S}}$ given by
	\begin{equation}\label{g}
	\begin{split}
		g(u):=\, & -\frac{\llbracket\sigma\rrbracket}{2(1+(\partial_x u(x))^2)} \big(\partial_x\psi_{u,2}+\partial_x u\partial_z\psi_{u,2}\big)^2 (x,u(x))\\
		&-\frac{\llbracket\sigma\rrbracket\sigma_2}{2\sigma_1(1+(\partial_x u(x))^2)}\big(\partial_x u\partial_x\psi_{u,2}-\partial_z\psi_{u,2}\big)^2 (x,u(x))\\
		&+\frac{\sigma_2}{2}\, \big\vert\nabla\psi_{u,2}(x,u(x)+d)\big\vert^2\,.
	\end{split}
\end{equation}
Finally, if $\llbracket\sigma\rrbracket<0$, then $u_*\le 0$ in $D$.
\end{theorem}

Even though the total energy $E$ consists of two competing terms with different signs, it is not difficult to see that it is $H^2$-coercive if $a>0$ in \eqref{E1}, see \cite{ARMA20}, and the existence of a minimizer for $E$ on $\bar{\mathcal{S}}_0$ follows directly. When $a=0$, the coercivity of $E$ is no longer obvious and one has to proceed differently. In this case, the coercivity of the functional can be enforced by adding a penalty term which vanishes when $u$ is bounded, an idea that was used in~\cite{CVPDE22}. The  minimizers of the penalized energy functional on $\bar{\mathcal{S}}_0$ then satisfy the Euler-Lagrange equation~\eqref{bennygoodman} with an additional term. The assumption $\llbracket\sigma\rrbracket<0$ now guarantees that $g(u)\ge 0$ in $D$ according to \eqref{g} which, in turn, yields an a priori bound on the minimizers by a comparison argument. This then implies that the minimizers of the penalized energy actually minimize the total energy $E$. It is worth emphasizing that the non-negative sign of $g(u)$ --~read off from the explicit formula~\eqref{g} when $\llbracket\sigma\rrbracket<0$~-- is essential for this approach.

The main motivation of this research is thus the explicit computation of the electrostatic force $g(u)$  as the (directional) derivative of the electrostatic energy $E_e(u)$. A computation in the same spirit was performed in~\cite{ARMA20} for a related MEMS model but with a flat transmission interface. As we shall see in \Cref{sec2}, the  non-flat transmission interface~$\Sigma(u)$ in~\eqref{a1b} makes the computation noticeably more involved. We first establish in \Cref{sec2} differentiability properties of the electrostatic potential $\psi_u$ with respect to $u$ which then ensure the Fr\'echet differentiability of the electrostatic energy $E_e$ on $\mathcal{S}$. The subsequent identification of $g(u)$ as the (directional) derivative of the electrostatic energy $E_e(u)$ is the main contribution of \Cref{sec2}. Once this is achieved, the existence of minimizers  follows along the lines of~\cite{CVPDE22} as described above.

As pointed out above the electrostatic force $g(u)$ has a sign if one assumes that $\llbracket\sigma\rrbracket<0$; that is, if $\sigma_2> \sigma_1$. This is a natural assumption e.g. if the region between the two plates is vacuumed or filled with air. We also point out that this assumption implies the monotonicity of the electrostatic energy $E_e$ as stated explicitly in  \Cref{C17c}.

\begin{remark}
The total energy $E$ can also be minimized in $\bar{\mathcal{S}}$ leading then to weak solutions to~\eqref{bennygoodman} with $\mathbb{I}_{\bar{\mathcal{S}}}$ instead of  $\mathbb{I}_{\bar{\mathcal{S}}_0}$ and pinned boundary conditions $u(\pm L) = \partial_x^2 u(\pm L)=0$ instead of the clamped boundary conditions involved in $\bar{\mathcal{S}}_0$ .
\end{remark}

\section{Shape Derivative of the Electrostatic Energy}\label{sec2}

The heart of the proof of \Cref{Thm3} is the differentiability of the electrostatic energy~$E_e$ and, in particular, the identification of $g(u)$ as its  derivative at $u\in{\bar{\mathcal{S}}}$.  On a formal level, this derivative is computed in \cite{LW18} (in a three-dimensional setting). Here we provide a rigorous proof.
Actually, we shall show that the electrostatic energy $E_e$  is   Fr\'echet differentiable  on
\begin{equation*}
	\mathcal{S} := \{u\in H^2(D) \cap H_0^1(D)\,:\, u> -H \text{ in } D \;\text{ and }\; \pm\llbracket\sigma\rrbracket  \partial_x u(\pm L) \le 0 \}\,,
\end{equation*}
i.e., for points with empty  coincidence set, while it admits a directional derivative at $u\in {\bar{\mathcal{S}}}$ in the directions $-u+\mathcal{S}$. Here and in the following, $\mathcal{S}$ and $\bar{\mathcal{S}}$ are endowed with the $H^2(D)$-topology. The precise result reads as follows:

\begin{theorem}\label{C17}
Assume \eqref{bobbybrown}. The electrostatic energy $E_e:\mathcal{S}\rightarrow\R$ is continuously Fr\'echet differentiable with
	$$
	\partial_uE_e(u)[\vartheta]= \int_D g(u)(x)\,\vartheta(x)\,\rd x
	$$
	for $u\in \mathcal{S}$ and $\vartheta\in H^2(D)\cap H_0^1(D)$, where $g(u)$ is defined in~\eqref{g}.
	Moreover, if $u\in\bar{\mathcal{S}}$ and $w\in \mathcal{S}$, then
	\begin{equation*}
		\begin{split}
			\lim_{t\rightarrow 0^+} \frac{1}{t}\big(E_e(&u+t(w-u))-E_e(u)\big)= \int_D g(u)(x)\,(w-u)(x)\,\rd x\,.
		\end{split}
	\end{equation*}
	The function $g:\bar{\mathcal{S}}\rightarrow L_p(D)$ is continuous for each $p\in [1,\infty)$.
\end{theorem}

The proof of \Cref{C17} follows from \Cref{C15} and \Cref{C17c} below. 
We will need the following result which is contained in \cite{LW21}.

\begin{proposition}{\bf\cite[Theorem~1.3, Proposition~3.3]{LW21}}\label{Thm2}
	Assume \eqref{bobbybrown}. 
	Let $u\in \bar{\mathcal{S}}$ and consider a bounded sequence $(u_n)_{n\ge 1}$ in $\bar{\mathcal{S}}$ such that 
	\begin{equation*}
		\lim_{n\to\infty} \|u_n-u\|_{H^1(D)} = 0\,. 
	\end{equation*}
	Then, for any $p\in [1,\infty)$, 
	\begin{subequations}\label{o1}
		\begin{align}
				\lim_{n\to\infty} \big\|\nabla\psi_{u_n,2}(\cdot,u_n) - \nabla\psi_{u,2}(\cdot,u) \big\|_{L_p(D,\R^2)} & = 0\,, \\	
				\lim_{n\to\infty} \big\| \nabla\psi_{u_n,2}(\cdot,u_n+d) - \nabla\psi_{u,2}(\cdot,u+d) \big\|_{L_p(D,\R^2)}  & = 0 \,.
			\end{align}
	\end{subequations}
	Moreover,
	\begin{equation} \label{o2}
		\lim_{n\to\infty}E_e(u_n)=E_e(u)\,.
	\end{equation}
	Finally, setting
	$$
	M := d + \max\big\{ \|u\|_{L_\infty(D)} \,,\, \sup_{n\ge 1}\{\|u_n\|_{L_\infty(D)}\}  \big\}\,, 
	$$
	one has
	\begin{equation} \label{o3}
		\lim_{n\to\infty} \left\| (\psi_{u_n}-h_{u_n}) - (\psi_u - h_u) \right\|_{H_0^1(D\times (-H,M))} = 0\,.
	\end{equation}
\end{proposition}

The first step of the proof of \Cref{C17} is to show that  the electrostatic energy~$E_e$  is Fr\'echet differentiable on $\mathcal{S}$. The next lemma is adapted from \cite[Theorem~5.3.2]{HP05}, see also \cite[Lemma~4.1]{ARMA20}. We include the proof for the reader's ease.

\begin{lemma}\label{P316}
Assume \eqref{bobbybrown}. Let $u\in \mathcal{S}$ be fixed and define, for $v\in \mathcal{S}$, the transformation 
\begin{equation*}
\Theta_{u,v}=(\Theta_{u,v,1},\Theta_{u,v,2}):\Omega(u)\rightarrow \Omega(v)
\end{equation*}  
by 
\begin{subequations}\label{tr}
\begin{align}
&\Theta_{u,v,1}(x,z):=\left(x,z+\frac{v(x)-u(x)}{H+u(x)}(z+H)\right)\,,& &(x,z)\in\Omega_1(u)\,,\\
&\Theta_{u,v,2}(x,z):=(x,z+v(x)-u(x))\,,& &(x,z)\in\Omega_2(u) \,.
\end{align}
\end{subequations}
Then there exists a neighborhood $\mathcal{U}$ of $u$ in $\mathcal{S}$ such that the mapping
\begin{equation*}
\mathcal{U}\rightarrow H_0^1(\Omega(u)),\quad v\mapsto \xi_v:=\big(\psi_v-h_v\big)\circ \Theta_{u,v}
\end{equation*}
is continuously differentiable, recalling that $\mathcal{S}$ and thus also  $\mathcal{U}$ are endowed with the $H^2(D)$-topology.
\end{lemma}

\begin{proof}
Set $\chi_v:=\psi_v-h_v$ for $v\in \mathcal{S}$. Owing to \Cref{Thm1}, the function $\chi_v$ belongs to $H_0^1(\Omega(v))$ and satisfies the integral identity
\begin{equation}\label{w1}
\int_{\Omega(v)}\sigma\nabla\chi_v\cdot\nabla \theta\,\rd(\bar x,\bar z)
=-\int_{\Omega(v)}\sigma\nabla h_v\cdot\nabla \theta\,\rd(\bar x,\bar z)\,,\quad \theta\in H_0^1(\Omega(v))\,,
\end{equation}
which we next shall write as integrals over $\Omega(u)$. To this end, we first note that
\begin{equation}
\xi_u=\chi_u\,,\qquad \nabla\xi_v=D\Theta_{u,v}^T\nabla\chi_v\circ \Theta_{u,v}\,, \label{gksi}
\end{equation}
where 
\begin{equation*}
 D\Theta_{u,v,1}(x,z)=\left(\begin{matrix}  1&0\\
\\
\displaystyle{(z+H)\partial_x\left(\frac{v-u}{H+u}\right)(x)} & \displaystyle{\frac{H+v(x)}{H+u(x)} }\end{matrix}\right)\,,\qquad (x,z)\in \Omega_1(u)\,,
\end{equation*}
and
\begin{equation*}
 D\Theta_{u,v,2}(x,z)=\left(\begin{matrix}  1&0\quad \\
\\
\partial_x(v-u)(x) & 1\quad \end{matrix}\right)\,,\qquad (x,z)\in \Omega_2(u)\,.
\end{equation*}
For $\phi\in H_0^1(\Omega(u))$ we set
\begin{equation*}
\phi_v:=\phi\circ \Theta_{u,v}^{-1}\in H_0^1(\Omega(v))
\end{equation*}
and note that  
\begin{equation*}
\nabla\phi_v=\big((D\Theta_{u,v}^T)^{-1}\nabla\phi\big)\circ \Theta_{u,v}^{-1}\,.
\end{equation*}
Performing the change of variables $(\bar x,\bar z)=\Theta_{u,v}(x,z)$ in \eqref{w1} with $\theta=\phi_v$ and using \eqref{sigma} give
\begin{equation}\label{w6}
\begin{split}
\int_{\Omega(u)} \sigma\, J_v  &(D\Theta_{u,v})^{-1} (D\Theta_{u,v}^T)^{-1}\nabla\xi_v\cdot\nabla\phi\,\rd (x,z)\\
&=  
-\int_{\Omega(u)} \sigma\, J_v\, (D\Theta_{u,v})^{-1}\nabla h_v\circ \Theta_{u,v}\cdot\nabla\phi\,\rd (x,z)\,,
\end{split}
\end{equation}
where the Jacobian $J_v:=\vert\mathrm{det}(D\Theta_{u,v})\vert$ is given by
\begin{equation}\label{Jv}
J_{v,1}=\frac{H+v}{H+u} \;\;\text{ in }\;\; \Omega_1(u)\,,\qquad  J_{v,2}= 1  \;\;\text{ in }\;\; \Omega_2(u)\,.
\end{equation}
Introducing the notations
\begin{equation*}
A(v):=\sigma\,J_v\,  (D\Theta_{u,v})^{-1} (D\Theta_{u,v}^T)^{-1}
\end{equation*}
and 
\begin{equation*}
B(v):=\mathrm{div}\big(\sigma\,J_v\, (D\Theta_{u,v})^{-1}\nabla h_v\circ \Theta_{u,v}\big)\,,
\end{equation*}
we define the function 
\begin{equation*}
F: \mathcal{S}\times H_0^1(\Omega(u))\rightarrow H^{-1}(\Omega(u))\,,\quad (v,\xi)\mapsto -\mathrm{div}\big(A(v)\nabla\xi\big)-B(v)
\end{equation*}
and observe that \eqref{w6} is equivalent to 
\begin{equation}\label{w7}
F(v,\xi_v)=0\,,\quad v\in \mathcal{S}\,.
\end{equation}
We then shall use the implicit function theorem to show that $\xi_v$ depends smoothly on~$v$. For that purpose, let us first show that $F$ is Fr\'echet differentiable in $\mathcal{S}\times H_0^1(\Omega(u))$. Indeed, by \eqref{bobbybrown}, it is readily checked that
\begin{equation*}
\nabla h_{v} \circ \Theta_{u,v} (x,z)= {\bf 1}_{\Omega_2(u)}\zeta'\big(z- u(x)+1\big) \left(\begin{array}{c} -\partial_x v(x) \\ 1 \end{array}\right)\,,
\end{equation*}
so that its Fr\'echet derivative with respect to $v$ is
\begin{equation}\label{diffh}
\partial_v\big(\nabla h_{v} \circ \Theta_{u,v}\big)[\vartheta] (x,z)={\bf 1}_{\Omega_2(u)}\zeta'\big(z- u(x)+1\big) \left(\begin{array}{c} -\partial_x \vartheta(x) \\ 0 \end{array}\right)
\end{equation}
for $\vartheta\in H^2(D)\cap H_0^1(D)$. Thus,
\begin{equation*}
\big[v\mapsto\nabla h_{v} \circ \Theta_{u,v}\big]\in C^1\big(\mathcal{S},L_2(\Omega(u),\R^2)\big)\,.
\end{equation*}
Moreover, $v\mapsto J_v$ and $v\mapsto (D\Theta_{u,v})^{-1}$ are continuously differentiable from  $\mathcal{S}$ to $L_\infty(\Omega(u))$ and $L_\infty(\Omega(u),\R^{2\times 2})$, respectively, and we conclude that
\begin{equation*}
v\mapsto \sigma\,J_v\, (D\Theta_{u,v})^{-1}\nabla h_v\circ \Theta_{u,v}
\end{equation*}
is continuously differentiable from $\mathcal{S}$ to $L_2(\Omega(u),\R^{2})$. Hence $B\in C^1(\mathcal{S},H^{-1}(\Omega(u)))$. The $C^1$-smoothness of $(v,\xi)\mapsto \mathrm{div}(A(v)\nabla\xi)$ is proven as in \cite[Theorem~5.3.2]{HP05} and we have thus established that
\begin{equation*}
F\in C^1\big(\mathcal{S}\times H_0^1(\Omega(u)), H^{-1}(\Omega(u))\big)\,.
\end{equation*}
The Lax-Milgram theorem and the open mapping theorem imply that the mapping 
$$
\omega\mapsto\partial_\xi F(u,\xi_u)[\omega] = -\mathrm{div}(\sigma\nabla\omega)
$$ 
is  an isomorphism from $H_0^1(\Omega(u))$ to $H^{-1}(\Omega(u))$. Consequently, the implicit function theorem ensures the existence of a neighborhood $\mathcal{U}$ of $u$ in $\mathcal{S}$ and a function $\Xi\in C^1(\mathcal{U},H_0^1(\Omega(u))$ such that 
$$
	\Xi(u)=\xi_u \;\;\text{ and }\;\; F(v,\Xi(v))=0 \;\;\text{ for }\;\; v\in \mathcal{U}\,.
$$ 
By~\eqref{o3}, $\xi_v\in \Xi(\mathcal{U})$ for $\|v-u\|_{H^2(D)}$ sufficiently small and we infer from \eqref{w7} and the uniqueness provided by the implicit function theorem that $\xi_v=\Xi(v)$ for $v\in \mathcal{U}$. 
\end{proof}

We next compute the Fr\'echet derivative of the electrostatic energy  on $\mathcal{S}$ and thus provide a proof for the first part of \Cref{C17}. 

\begin{proposition}\label{C15}
Assume \eqref{bobbybrown}. The electrostatic energy $E_e:\mathcal{S}\rightarrow\R$ is continuously Fr\'echet differentiable with
$$
\partial_uE_e(u)[\vartheta]= \int_D g(u)(x)\,\vartheta(x)\,\rd x
$$
for $u\in \mathcal{S}$ and $\vartheta\in H^2(D)\cap H_0^1(D)$, where $g(u)$ is defined in~\eqref{g}.
\end{proposition}

\begin{proof}
We fix $u\in \mathcal{S}$ and use the notation introduced in \Cref{P316}. Recall that, according to \Cref{P316}, there is a neighborhood $\mathcal{U}$ of $u$ in $\mathcal{S}$ such that the mapping  
$$
v\mapsto \xi_v=\big(\psi_v-h_v\big)\circ \Theta_{u,v}
$$  
belongs to $C^1(\mathcal{U},H_0^1(\Omega(u)))$, the transformation $\Theta_{u,v}:\Omega(u)\rightarrow \Omega(v)$
being defined in~\eqref{tr}. Now, for $v\in\mathcal{U}$, we use \eqref{gksi}, the relation $\chi_v=\psi_v-h_v$, and the change of variable $(\bar x,\bar z)=\Theta_{u,v}(x,z)$ in the integral defining $E_e(v)$ to obtain
\begin{equation*}
E_e(v) = -\frac{1}{2}\int_{\Omega(v)} \sigma \vert\nabla \psi_v\vert^2\,\rd (\bar x,\bar z) = -\frac{1}{2} \int_{\Omega(u)} \sigma |j(v)|^2 J_v\,\rd (x,z)\,,
\end{equation*}
where
\begin{equation*}
j(v):=(D\Theta_{u,v}^T)^{-1}\nabla\xi_v+\nabla h_v\circ\Theta_{u,v}\,.
\end{equation*}
Owing to the differentiability of $v\mapsto \xi_v$ in $\mathcal{U}$, we deduce that the Fr\'echet derivative of $E_e$ at $u$ applied to some $\vartheta\in H^2(D)\cap H_0^1(D)$ is given by
\begin{equation*}
\begin{split}
\partial_u E_e(u)[\vartheta]=\partial_v E_e(v)[\vartheta]\big\vert_{v=u}=\, & -\int_{\Omega(u)} \sigma  j(u)\cdot (\partial_v j(v))[\vartheta]\big\vert_{v=u}\, J_u\,\rd (x,z)\\
&  -\frac{1}{2}\int_{\Omega(u)} \sigma \vert j(u)\vert^2\, (\partial_v J_v)[\vartheta]\big\vert_{v=u}\,\rd (x,z)\,.
\end{split}
\end{equation*}
Taking the identity $j(u) =\nabla\chi_u + \nabla h_u=\nabla\psi_u$ into account, we infer from \eqref{Jv} that
\begin{equation}\label{D1}
\begin{split}
\partial_u E_e(u)[\vartheta]=\, & -\int_{\Omega(u)} \sigma  \nabla\psi_u\cdot \big(\partial_v j(v)[\vartheta]\big\vert_{v=u}\big)\,\rd (x,z)\\
& -\frac{1}{2}\int_{\Omega_1(u)} \sigma_1 \vert \nabla\psi_{u,1}\vert^2\, \frac{\vartheta}{H+u}\,\rd (x,z)\,.
\end{split}
\end{equation}
We next use that $\Theta_{u,u}$ is the identity on $\Omega(u)$ and that $\xi_u=\chi_u$ to compute from the definition of $j(v)$ that
\begin{equation}\label{D1a}
\begin{split}
\partial_v j(v)[\vartheta]\big\vert_{v=u}=\, &-\partial_v (D\Theta_{u,v}^T)[\vartheta]\big\vert_{v=u} \nabla \chi_u + \partial_v (\nabla\xi_v)[\vartheta]\big\vert_{v=u}\\
& +\partial_v (\nabla h_v\circ \Theta_{u,v})[\vartheta]\big\vert_{v=u}\,.
\end{split}
\end{equation}
Now, $\chi_{u,1}=\psi_{u,1}$ in $\Omega_1(u)$ due to \eqref{bobbybrown}, so that
\begin{equation}\label{D3}
-\partial_v (D\Theta_{u,v}^T)[\vartheta]\big\vert_{v=u} \nabla \chi_u=-\partial_z\psi_{u}\nabla\left(\frac{\vartheta(z+H)}{H+u}\right) \quad \text{ in }\ \Omega_1(u)\,,
\end{equation}
while
\begin{equation}\label{D2}
-\partial_v (D\Theta_{u,v}^T)[\vartheta]\big\vert_{v=u} \nabla \chi_u=-\left(\begin{array}{c} \partial_z\chi_u \partial_x\vartheta\\ 0\end{array}\right)\quad \text{ in }\ \Omega_2(u)\,.
\end{equation}
Also note that
\begin{equation}\label{D4}
\partial_v (\nabla\xi_v)[\vartheta]\big\vert_{v=u}=\nabla \big(\partial_v \xi_v[\vartheta]\big\vert_{v=u}\big)\quad \text{ in }\ \Omega(u)\,.
\end{equation}
Consequently, gathering \eqref{D1}-\eqref{D4} and recalling~\eqref{diffh} lead us to
\begin{equation}\label{D6}
\partial_u E_e(u)[\vartheta] = I_0(u)[\vartheta] + I_1(u)[\vartheta] + I_2(u)[\vartheta]\,,
\end{equation}
where
\begin{equation*}
	I_0(u)[\vartheta] := -\int_{\Omega(u)} \sigma\,  \nabla\psi_u\cdot \nabla\big(  \partial_v \xi_v[\vartheta]\big\vert_{v=u}\big)\,\rd (x,z)\,,
\end{equation*}
\begin{align*}
	I_1(u)[\vartheta]  :=\, & \int_{\Omega_1(u)} \sigma_1\, \partial_z\psi_{u,1}\,  \nabla\psi_{u,1}\cdot \nabla \left(\frac{\vartheta(z+H)}{H+u}\right)\,\rd (x,z)\\
	&   -\frac{1}{2}\int_{\Omega_1(u)} \sigma_1\,  \vert \nabla\psi_{u,1}\vert^2 \,  \frac{\vartheta}{H+u} \,\rd (x,z)\,,
\end{align*}
and
\begin{align*}
	I_2(u)[\vartheta]  :=\, & \int_{\Omega_2(u)} \sigma_2\,  \partial_x\psi_{u,2}\, \zeta'\big(z- u+1\big)\, \partial_x\vartheta \,\rd (x,z)\\
	&  +\int_{\Omega_2(u)} \sigma_2\, \partial_x\psi_{u,2}\,\partial_z\chi_{u,2}\, \partial_x\vartheta\, \rd (x,z)\,.
\end{align*}
We are left with simplifying these three integrals and begin with $I_0(u)[\vartheta]$. We use Gau\ss ' theorem and \eqref{a1a}  to get
\begin{equation*}
\begin{split}
I_0(u)[\vartheta]=\, & -\int_{\partial\Omega(u)} \big(  \partial_v \xi_v[\vartheta]\big\vert_{v=u} \big)\sigma\nabla\psi_u\cdot  {\bf n}_{\partial\Omega(u)}\,\rd S\\
&-\int_{\Sigma(u)}\llbracket \partial_v \xi_v[\vartheta]\big\vert_{v=u} \sigma\nabla\psi_u\rrbracket\cdot {\bf n}_{\Sigma(u)}\,\rd S\,.
\end{split}
\end{equation*}
Now, recall that $\partial_v \xi_v[\vartheta]\big\vert_{v=u}$ belongs to $H_0^1(\Omega(u))$ according to \Cref{P316}. On the one hand, this entails that $\partial_v \xi_v[\vartheta]\big\vert_{v=u}$ vanishes on $\partial\Omega(u)$, so that the first integral on the right-hand side of the above identity is zero. On the other hand, the $H^1$-regularity of $\partial_v \xi_v[\vartheta]\big\vert_{v=u}$ also implies that $\llbracket \partial_v \xi_v[\vartheta]\big\vert_{v=u} \rrbracket=0$ on $\Sigma(u)$, so that
\begin{equation*}
\llbracket \partial_v \xi_v[\vartheta]\big\vert_{v=u} \sigma\nabla\psi_u\rrbracket\cdot {\bf n}_{\partial\Sigma(u)}= \partial_v \xi_v[\vartheta]\big\vert_{v=u} \,\llbracket\sigma\nabla\psi_u\rrbracket\cdot {\bf n}_{\Sigma(u)}=0\quad\text{on }\ \Sigma(u)
\end{equation*}
due to \eqref{a1b}. Therefore,
\begin{equation}\label{i0}
I_0(u)[\vartheta] = 0 \,.
\end{equation}
We next deal with $I_1(u)[\vartheta]$. Since $\sigma_1 \Delta\psi_{u,1} = \mathrm{div}(\sigma\nabla\psi_u) = 0$ in $\Omega_1(u)$ by \eqref{a1a}, it follows from Gau\ss' theorem that
\begin{align*}
	I_1(u)[\vartheta]  =\, & \int_{\Omega_1(u)} \sigma_1\, \partial_z\psi_{u,1}\, \mathrm{div}\left( \left(\frac{\vartheta(z+H)}{H+u}\right) \nabla\psi_{u,1} \right)\,\rd (x,z)\\
	&   -\frac{1}{2}\int_{\Omega_1(u)} \sigma_1\,  \vert \nabla\psi_{u,1}\vert^2 \,  \frac{\vartheta}{H+u} \,\rd (x,z) \\
	 =\, & \int_{\partial\Omega_1(u)}\sigma_1\, \frac{\vartheta(z+H)}{H+u} \partial_z\psi_{u,1}\nabla\psi_{u,1}\cdot {\bf n}_{\partial\Omega_1(u)}\,\rd S\\
	&  - \int_{\Omega_1(u)} \sigma_1\,  \nabla\psi_{u,1}\cdot \nabla\left(\partial_z\psi_{u,1}\,\right)\frac{\vartheta(z+H)}{H+u}\,\rd (x,z)\\
	&   -\frac{1}{2}\int_{\Omega_1(u)} \sigma_1\,  \vert \nabla\psi_{u,1}\vert^2 \,  \frac{\vartheta}{H+u} \,\rd (x,z) \,.
\end{align*}
Recalling that $\vartheta\in H_0^1(D)$ and noticing that $\nabla\psi_{u,1}\cdot \nabla\left(\partial_z\psi_{u,1}\,\right)=\partial_z\big(|\nabla\psi_{u,1}|^2)/2$, we further obtain
\begin{align*}
	I_1(u)[\vartheta]  =\, &  \int_{D} \sigma_1\, \partial_z \psi_{u,1}(x,u(x)) \big( - \partial_x u \partial_x \psi_{u,1} + \partial_z \psi_{u,1} \big)(x,u(x)) \vartheta(x) \,\rd x \\
	&  - \frac{1}{2} \int_D \sigma_1\, |\nabla\psi_{u,1}(x,u(x))|^2 \vartheta(x)\,\rd x \,.
\end{align*}
Hence, 
\begin{equation}
	\begin{split}
	I_1(u)[\vartheta]  =\, & -\frac{1}{2} \int_{D} \sigma_1\, \left( |\partial_x\psi_{u,1}|^2 - |\partial_z\psi_{u,1}|^2 \right)(x,u(x)) \vartheta(x)\,\rd x  \\
	&  - \int_{D} \sigma_1\, \partial_x u(x)\big( \partial_x\psi_{u,1} \partial_z\psi_{u,1} \big)(x,u(x)) \vartheta(x) \,\rd x\,.
	\end{split}\label{i1}
\end{equation}
Finally, using \eqref{exx1}, $\chi_u=\psi_u-h_u$ and $\vartheta\in H_0^1(D)$, it follows from Green's formula that 
\begin{align*}
	I_2(u)[\vartheta]  =\, &  \int_{\Omega_2(u)} \sigma_2\, \partial_x\psi_{u,2} \partial_z\psi_{u,2} \partial_x\vartheta \,\rd (x,z) \\
	 =\, & -\int_{D} \sigma_2\, \big( \partial_x\psi_{u,2} \partial_z\psi_{u,2} \big)(x,u(x)+d) \partial_x u(x)\,\rd x\\
	&  + \int_{D} \sigma_2\, \big( \partial_x\psi_{u,2} \partial_z\psi_{u,2} \big)(x,u(x)) \partial_x u(x)\,\rd x\\
	&  - \int_{\Omega_2(u)} \sigma_2\, \partial_x \big( \partial_x\psi_{u,2} \partial_z\psi_{u,2}\big) \vartheta \,\rd (x,z) \,.
\end{align*}
Owing to \eqref{a1a}, we have $\sigma_2 \partial_x^2\psi_{u,2} = -\sigma_2 \partial_z^2\psi_{u,2}$ in $\Omega_2(u)$ from which we deduce that
\begin{align*}
	& \int_{\Omega_2(u)} \sigma_2\, \partial_x \big( \partial_x\psi_{u,2} \partial_z\psi_{u,2}\big) \vartheta \,\rd (x,z) \\
	& \qquad\qquad = \int_{\Omega_2(u)} \sigma_2\, \big( \partial_x^2\psi_{u,2} \partial_z\psi_{u,2} + \partial_x\psi_{u,2} \partial_x\partial_z\psi_{u,2} \big) \vartheta \,\rd (x,z) \\
	& \qquad\qquad = \int_{\Omega_2(u)} \sigma_2\, \big( - \partial_z\psi_{u,2} \partial_z^2\psi_{u,2}+ \partial_x\psi_{u,2} \partial_x\partial_z\psi_{u,2} \big) \vartheta \,\rd (x,z) \\
	& \qquad\qquad = \frac{1}{2} \int_{\Omega_2(u)} \sigma_2\, \partial_z \big( |\partial_x\psi_{u,2}|^2 - |\partial_z\psi_{u,2}|^2 \big) \vartheta \,\rd (x,z) \\
	& \qquad\qquad = \frac{1}{2} \int_{D} \sigma_2\, \big( |\partial_x\psi_{u,2}|^2 - |\partial_z\psi_{u,2}|^2 \big)(x,u(x)+d) \vartheta(x) \,\rd x \\
	& \qquad\qquad\quad  - \frac{1}{2} \int_{D} \sigma_2\, \big( |\partial_x\psi_{u,2}|^2 - |\partial_z\psi_{u,2}|^2 \big)(x,u(x)) \vartheta(x) \,\rd x\,.
\end{align*}
Consequently,
\begin{align*}
	I_2(u)[\vartheta]  =\, & -\int_{D} \sigma_2\, \big( \partial_x\psi_{u,2} \partial_z\psi_{u,2} \big)(x,u(x)+d) \partial_x u(x)\,\rd x\\
	&  + \int_{D} \sigma_2\, \big( \partial_x\psi_{u,2} \partial_z\psi_{u,2} \big)(x,u(x)) \partial_x u(x)\,\rd x\\
	&  - \frac{1}{2} \int_{D} \sigma_2\, \big( |\partial_x\psi_{u,2}|^2 - |\partial_z\psi_{u,2}|^2 \big)(x,u(x)+d) \vartheta(x) \,\rd x \\
	&  + \frac{1}{2} \int_{D} \sigma_2\, \big( |\partial_x\psi_{u,2}|^2 - |\partial_z\psi_{u,2}|^2 \big)(x,u(x)) \vartheta(x) \,\rd x\,.
\end{align*}
We finally note that
\begin{equation*}
	\partial_x\psi_{u,2}(x,u(x)+d))=-\partial_x u(x) \partial_z\psi_{u,2}(x,u(x)+d)\,,
\end{equation*}
since $\psi_{u,2}(x,u(x)+d)=V$ owing to \eqref{a1c} and \eqref{z1}. This identity allows us to simplify further the formula for $I_2(u)[\vartheta]$, so that we end up with
\begin{equation}
	\begin{split}
		I_2(u)[\vartheta]  =\, &  \frac{1}{2} \int_{D} \sigma_2\, |\nabla\psi_{u,2}(x,u(x)+d)|^2\, \rd x \\
		&  + \int_{D} \sigma_2\, \big( \partial_x\psi_{u,2} \partial_z\psi_{u,2} \big)(x,u(x)) \partial_x u(x)\,\rd x\\
		 &  + \frac{1}{2} \int_{D} \sigma_2\, \big( |\partial_x\psi_{u,2}|^2 - |\partial_z\psi_{u,2}|^2 \big)(x,u(x)) \vartheta(x) \,\rd x\,.
	\end{split}\label{i2}
\end{equation} 
Collecting \eqref{D6}, \eqref{i0}, \eqref{i1}, and \eqref{i2} gives
\begin{equation}\label{p10}
\begin{split}
\partial_u E_e(u)[\vartheta]=\, & - \frac{1}{2}\int_D \left\llbracket \sigma(\partial_x\psi_u)^2 -\sigma(\partial_z\psi_u)^2 \right\rrbracket (x,u(x))\,\vartheta(x)\,\rd x\\
&-\int_D \partial_x u(x)\,\left\llbracket \sigma\partial_x\psi_u\partial_z\psi_u  \right\rrbracket (x,u(x))\,\vartheta(x)\,\rd x\\
&+\frac{1}{2} \int_{D}\sigma_2\, \big\vert\nabla\psi_{u,2}(x,u(x)+d)\big\vert^2\,\vartheta(x)\,\rd x\,.
\end{split}
\end{equation}
Finally, we shall write \eqref{p10} only in terms of $\psi_{u,2}$. To this end, we set
\begin{equation}
	F_u := \partial_x \psi_u + \partial_x u \partial_z \psi_u\,, \qquad G_u := - \partial_x u \partial_x \psi_u + \partial_z \psi_u\,, \label{FG}
\end{equation}
and observe that differentiating the transmission condition $\llbracket \psi_u \rrbracket = 0$ on $\Sigma(u)$, along with the second transmission condition in \eqref{a1b}, ensures that
\begin{equation*}
	\llbracket F_u \rrbracket = \llbracket \sigma G_u \rrbracket = 0 \;\;\text{ on }\;\; \Sigma(u)\,.
\end{equation*}
These properties in turn imply that
\begin{equation}
	\llbracket \sigma F_u^2 \rrbracket = \llbracket \sigma\rrbracket F_{u,2}^2\,, \qquad  \llbracket \sigma F_u G_u \rrbracket = 0\,, \qquad \llbracket \sigma G_u^2 \rrbracket = \left\llbracket \frac{1}{\sigma} \right\rrbracket \sigma_2^2 G_{u,2}^2 \;\;\text{ on }\;\; \Sigma(u)\,. \label{VH}
\end{equation}
Guided by \eqref{VH}, we next express the jump terms in \eqref{p10} using $F_u$ and $G_u$. Since
\begin{equation*}
	\left[ 1+ (\partial_x u)^2\right] \partial_x \psi_u  = F_u - G_u \partial_x u \quad \text{ and }\quad \left[ 1+ (\partial_x u)^2\right] \partial_z \psi_u = F_u \partial_x u + G_u\,, 
\end{equation*}
we compute
\begin{align*}
	& \left[ 1+ (\partial_x u)^2\right]^2 \left[ (\partial_x\psi_u)^2 - (\partial_z\psi_u)^2 + 2 \partial_x u \partial_x\psi_u \partial_z\psi_u \right] \\
	& \quad\qquad\qquad = (F_u-G_u \partial_x u)^2 - (F_u \partial_x u + G_u)^2 + 2 \partial_x u (F_u - G_u \partial_x u) (F_u \partial_x u + G_u)  \\
	& \quad\qquad\qquad = \left[ 1+ (\partial_x u)^2\right] \left( F_u^2 - 2 F_u G_u \partial_x u - G_u^2 \right) \,.
\end{align*}
Therefore, by \eqref{VH},
\begin{align*}
	& \left[ 1+ (\partial_x u)^2\right] \big\llbracket \sigma (\partial_x\psi_u)^2 - \sigma (\partial_z\psi_u)^2 + 2 \sigma \partial_x u \partial_x\psi_u \partial_z\psi_u \big\rrbracket \\
	& \hspace{2cm} = \big\llbracket \sigma F_u^2 - 2 \sigma F_u G_u \partial_x u - \sigma G_u^2 \big\rrbracket = \llbracket\sigma\rrbracket F_{u,2}^2 - \left\llbracket \frac{1}{\sigma} \right\rrbracket \sigma_2^2 G_{u,2}^2 \\
	& \hspace{2cm} = \llbracket\sigma\rrbracket F_{u,2}^2 +  \frac{\llbracket\sigma\rrbracket\sigma_2}{\sigma_1} G_{u,2}^2\,.
\end{align*}
Consequently, plugging this formula into \eqref{p10} and recalling \eqref{FG} yield
\begin{equation*}
\begin{split}
\partial_u E_e (u)[\vartheta]=\, & - \frac{\llbracket\sigma\rrbracket}{2}\int_D\frac{1}{1+(\partial_x u(x))^2} \big(\partial_x\psi_{u,2}+\partial_xu(x)\partial_z\psi_{u,2}\big)^2 (x,u(x))\,\vartheta(x)\,\rd x\\
&-\frac{\llbracket\sigma\rrbracket\sigma_2}{2\sigma_1}\int_D\frac{1}{1+(\partial_x u(x))^2}\big(\partial_x u(x)\partial_x\psi_{u,2}-\partial_z\psi_{u,2}\big)^2 (x,u(x))\,\vartheta(x)\,\rd x\\
&+\frac{1}{2} \int_{D}\sigma_2\, \big\vert\nabla\psi_{u,2}(x,u(x)+d)\big\vert^2\,\vartheta(x)\,\rd x\,;
\end{split}
\end{equation*}
that is,
$$
\partial_uE_e(u)[\vartheta]= \int_D g(u)(x)\,\vartheta(x)\,\rd x
$$
for $u\in \mathcal{S}$ and $\vartheta\in H^2(D)\cap H_0^1(D)$ with $g(u)$ being defined in~\eqref{g}. It then readily follows from~\eqref{o1} that  
$$
 \partial_uE_e:\mathcal{S}\rightarrow \mathcal{L}\big( H^2(D)\cap H_0^1(D),\R\big)
 $$
is continuous.
\end{proof}

The final step for the proof of \Cref{C17} is to show that the electrostatic energy~$E_e$ admits directional derivatives in the directions $-u+\mathcal{S}$.

\begin{corollary}\label{C17c}
Assume \eqref{bobbybrown}. Let $u_0\in\bar{\mathcal{S}}$ and $u_1\in \mathcal{S}$. Then
\begin{equation*}
\begin{split}
\lim_{t\rightarrow 0^+} \frac{1}{t}\big[ E_e(&u_0+t(u_1-u_0))-E_e(u_0) \big] = \int_D g(u_0)(x)\,(u_1-u_0)(x)\,\rd x\,.
\end{split}
\end{equation*}
Moreover, the function $g:\bar{\mathcal{S}}\rightarrow L_p(D)$ is continuous for each $p\in [1,\infty)$.
\end{corollary}

\begin{proof}
The stated continuity of $g$ is a straightforward consequence of \eqref{o1}. Next, given $u_0\in\bar{\mathcal{S}}$ and $u_1\in \mathcal{S}$, we set
\begin{equation*}
u_s:= u_0+s(u_1-u_0)=(1-s)u_0+su_1 \in \mathcal{S}\,,\qquad s\in (0,1]\,.
\end{equation*}
 Since $u_s\in \mathcal{S}$ for $s\in (0,1]$, we deduce from \Cref{C15} that
\begin{equation}\label{bn1}
\begin{split}
\frac{\rd}{\rd s} E_e(u_s) =\, & \int_D g(u_s)(x)\, (u_1-u_0)(x)\,\rd x\,,\qquad s\in (0,1]\,.
\end{split}
\end{equation}
Therefore, letting $s\rightarrow 0$, the continuity of $g$ entails 
\begin{equation}\label{pppl}
\begin{split}
\lim_{s\rightarrow 0^+}\frac{\rd}{\rd s} E_e(u_s)= \int_D g(u_0)(x)\, (u_1-u_0)(x)\,\rd x\,.
\end{split}
\end{equation}
Now \eqref{o2}  guarantees that $E_e(u_s) \rightarrow E_e(u_0)$ as $s\rightarrow 0$,  so that
\begin{equation}\label{bn2}
E_e(u_t)-E_e(u_0)= \int_0^t \frac{\rd}{\rd s} E_e(u_s)\,\rd s\,,\quad t\in (0,1]\,,
\end{equation}
and we conclude from \eqref{pppl} that
\begin{equation*}
\begin{split}
\lim_{t\rightarrow 0^+} \frac{1}{t}\big(E_e(u_t)-E_e(u_0)\big)&= \lim_{t\rightarrow 0^+} \frac{1}{t}\int_0^t \frac{\rd}{\rd s} E_e(u_s)\,\rd s=\int_D g(u_0)(x) \, (u_1-u_0)(x)\,\rd x
\end{split}
 \end{equation*}
as claimed.
\end{proof}

If $\llbracket\sigma\rrbracket<0$, then an obvious consequence of \eqref{g} is that $g$ is non-negative on $\bar{\mathcal{S}}$. This yields the monotonicity of the electrostatic energy~$E_e$.

\begin{corollary}\label{C17b}
Assume $\llbracket\sigma\rrbracket<0$ and let \eqref{bobbybrown} be satisfied. If $u_0\in\bar{\mathcal{S}}$ and $u_1\in \mathcal{S}$ are such that $u_0\le u_1$ in $D$, then \mbox{$E_e(u_0)\le E_e(u_1)$}.
\end{corollary}
	
\begin{proof}
The assumption $\llbracket\sigma\rrbracket<0$ implies that $g(u_s)\ge 0$ for $s\in (0,1]$ according to \eqref{g}, where $u_s= (1-s) u_0 + s u_1$ as in the proof of \Cref{C17c}. Hence, \eqref{bn1} and \eqref{bn2} with $t=1$ imply the assertion.
\end{proof}

\section{Proof of \Cref{Thm3}}\label{sec3}

The proof of \Cref{Thm3} now follows from \Cref{C17} as in \cite{CVPDE22}. Indeed,
\Cref{C17} guarantees that any minimizer of the total energy $E$ on $\bar{\mathcal{S}}_0$ satisfies the Euler-Lagrange equation~\eqref{bennygoodman}. In case that $a>0$, the total energy $E$ is coercive and thus the existence of a minimizer of $E$ on $\bar{\mathcal{S}}_0$ can be shown  as in \cite[Section~7]{CVPDE22}.
In the more complex case $a=0$, the total energy $E$ need not be coercive. But, as pointed out in the introduction, one may enforce its coercivity by adding a penalizing term and proceed along the lines of \cite[Section~6]{CVPDE22}, recalling that the assumption $\llbracket\sigma\rrbracket<0$ guarantees that $g(u)\ge 0$ in $D$ which is essential in this case (see, in particular, \cite[Equation~(6.4)]{CVPDE22}).

\section*{Acknowledgments}

Part of this work was done while PhL enjoyed the hospitality and support of the Institut f\" ur Angewandte Mathematik, Leibniz Universit\"at Hannover.


\bibliographystyle{siam}
\bibliography{MEMSBIB}

\end{document}